\documentclass[11pt]{article}

\usepackage{amsmath,amssymb,amsthm,graphicx,url,array,stmaryrd}

\newtheorem{thm}{Theorem}
\newtheorem{lem}[thm]{Lemma}
\newtheorem{cor}[thm]{Corollary}
\newtheorem{conj}[thm]{Conjecture}

\renewcommand{\le}{\leqslant}
\renewcommand{\ge}{\geqslant}
\def\st{\,\vert\ }
\def\S{{\mathcal S}''}
\def\Se{{\mathcal S}'}
\def\eps{\varepsilon}

\begin{document}

\title{On the Set of Circular Total Chromatic Numbers of Graphs}
\author{Mohammad Ghebleh\\[2mm]
Department of Mathematics\\Kuwait University, State of Kuwait\\{\tt mamad@sci.kuniv.edu.kw}}
\date{}

\maketitle

\begin{abstract}
For every integer $r\ge3$ and every $\eps>0$ we construct a graph with maximum degree $r-1$
whose circular total chromatic number is in the interval $(r,r+\eps)$.
This proves that (i) every integer $r\ge3$ is an accumulation point of the set of circular total
chromatic numbers of graphs, and (ii) for every $\Delta\ge2$,
the set of circular total chromatic numbers of graphs with maximum degree~$\Delta$ is infinite.
All these results hold for the set of circular total chromatic numbers of bipartite graphs as well.
\end{abstract}

\section{Introduction}

Total colouring has attracted a considerable amount of attention since the infamous 
Total Colouring Conjecture (TCC) was proposed by Behzad~\cite{Behzad} and also by Vizing~\cite{VizingTCC}.
The TCC asks whether every graph with maximum degree $\Delta$ admits a $\Delta+2$--total colouring, namely
an assignment of one of $\Delta+2$ colours to each vertex and to each edge of the graph
so that there is no monochromatic pair of adjacent vertices, pair of adjacent edges, or
pair of incident vertex and edge.

We study the circular version of total colouring in this paper. Since a total colouring of a
graph $G$ is in turn a vertex colouring of the total graph of $G$, all basic properties of
circular colourings carry over to circular total colouring. Circular total colourings have been
studied in only a few articles so far, for example~\cite{BodeKemnitzKlages,HackmannKemnitz,HackmannKemnitz:circulant}.
For an introduction to circular colouring, as well as a survey of results we refer the reader
to~\cite{Zhu-survey} and~\cite{zhu:newsurvey}.

In this work we study the set $\S=\left\{\chi''_c(G)\st G\text{ is a graph}\right\}$, where
$\chi''_c(G)$ denotes the circular total chromatic number of the graph $G$.
It is an easy observation to see that $\S\cap(-\infty,3)=\{1,2\}$.
In terms of other known values of $\chi''_c(G)$, they are computed precisely for all graphs with maximum degree at most~$2$,
as well as graphs of order at most $7$, complete graphs, complete bipartite graphs, and some subcubic cases.

In particular it is proved in~\cite{HackmannKemnitz:circulant} that the M\"obius ladders $V_{2n}$ have
circular total chromatic number at most $4+\frac{1}{2}$ when $n\ge 4$, and that the prism
$K_2\boxempty C_5$ has circular total chromatic number at most $4+\frac{1}{3}$.
We verified by the aid of a computer program that indeed $\chi''_c(K_2\boxempty C_5)=4+\frac{1}{3}$
and $\chi''_c(V_{2n})=4+\frac{1}{2}$ for $n=4,5,6$.

In~\cite{BodeKemnitzKlages} an infinite family of graphs with circular total chromatic number at most
$4+\frac{1}{3}$ and a graph with circular total chromatic number at most $4+\frac{1}{4}$ are given.
While it follows from the results of~\cite{HackmannKemnitz} that the set of circular total chromatic numbers
of graphs with maximum degree~$2$ is precisely the set
$$\left\{3\right\}\cup\left\{3+{1}/{n}\st n\text{ is a positive integer}\right\},$$
the only non-integer numbers $a>4$ realized as the circular total chromatic number of some graph
are $4+\frac{1}{3}$, $4+\frac{1}{2}$, $5+\frac{1}{2}$ and $6+\frac{1}{2}$ (see~\cite{HackmannKemnitz:circulant}
and the above discussion).
In particular, if $k\ge3$, it is not known yet whether the set $\S_k=\left\{\chi''_c(G)\st\Delta(G)=k\right\}$ is infinite.

In this work we prove that every integer $r\ge 3$ is an 
accumulation point of~$\S_{r-1}$ and hence of $\S$. Thus $\S_k$ is infinite when $k\ge 2$. This work was partly motivated by the
following conjecture of Xuding Zhu for circular edge colouring:

\begin{conj}{\rm\cite{Zhu:accumulation}}
Integers $r\ge3$ are the only accumulation points of the set
$\Se=\left\{\chi'_c(G)\st G\text{ is a graph}\right\}$.
\label{conj:zhu}
\end{conj}

We conjecture that the same holds for circular total colouring.

\section{The Main Result}

For convenience, we allow half-edges in graphs. A half-edge has only one end vertex.
We may join two half-edges to form an edge joining their end vertices.
We use the notion of $(p,q)$--colourings for our upper bounds. A $(p,q)$--colouring of a graph $G$
is any function $c:V(G)\to\{0,1,\ldots,p-1\}$ such that for any pair $x,y$ of adjacent vertices
we have $q\le|c(x)-c(y)|\le p-q$. A $(p,q)$--total colouring of a graph is a $(p,q)$--colouring
of its total graph. The circular total chromatic number of a graph $G$ is defined by
\[\chi''_c(G)=\inf\left\{\frac{p}{q}\st G\text{ admits a }(p,q)\text{--total colouring}\right\}.\]
Therefore existence of a $(p,q)$--total colouring for $G$ implies $\chi''_c(G)\le p/q$.

Let $H_k$ be the graph obtained from the complete bipartite graph $K_{r,r}$ by deleting one vertex, but leaving its
incident edges in as half-edges.
The following lemma is the main ingredient in our construction.

\begin{lem}
The graph $H_k$ is type~$1$. Moreover, in every $k+1$--total colouring of $H_k$,
all the half-edges receive the same colour.
\label{lem:all0}
\end{lem}

\begin{proof}
Let $X=\{x_1,x_2,\ldots,x_k\}$ and $Y=\{y_2,\ldots,y_k\}$ be the partite sets of~$H_k$,
and let $e_i$ be the half-edge in $H_k$ incident with $x_i$. Throughout this proof $i$ and $j$ are
integers with $1\le i\le k$ and $2\le j\le k$. 
Let $L=[\ell_{ij}]$ be a Latin square of order~$k$ with entries from the set $\{1,\ldots,k\}$.
We define a $k+1$--total colouring $c$ of $H_k$ as follows: $c(x_iy_j)=\ell_{ij}$,
$c(x_i)=\ell_{i1}$, and $c(e_i)=c(y_j)=0$. The fact that $c$ is a proper total colouring of $H_k$ follows
immediately since $L$ is a Latin square.

For the second part of the lemma,
let $c$ be a $k+1$--total colouring of $H_k$ with the colours $0,1,\ldots,k$.
We may assume that $c(x_1)=1$, thus $c(y_j)\not=1$ for all~$j$.
On the other hand, since each $y_j$ has $k$ incident edges, all colours must appear
either on $y_j$ itself or on an edge incident with it. In particular each $y_j$ is incident with an edge
coloured~$1$. Therefore the set of edges coloured $1$ is a matching which saturates $Y$, hence it
also saturates $X\setminus\{x_1\}$. This means that $x_1$ is the only vertex coloured~$1$. Subsequently
the vertices in $X$ must all receive different colours, say $c(x_i)=i$ for $1\le i\le k$, which leaves only
the colour~$0$ for the vertices $y_j$. That is $c(y_j)=0$ for all~$j$.
Now the only missing colour at any $x_i$ is the colour $0$, so $c(e_i)=0$.
\end{proof}

The next lemma is a circular relaxation of Lemma~\ref{lem:all0}.

\begin{lem}
\def\c{\gamma}
Let $k\ge 2$. Let $e$ and $e'$ be two half-edges of the graph $H_k$, and let $x$ and $x'$ be the
vertices of $H_k$ incident with $e$ and $e'$ respectively.
Then for every positive integer $n$, $H_k$ admits an $(n(k+1)+1,n)$--total colouring $\c$
such that $\c(e)=0$, $\c(e')=1$, $\c(x)=n+1$, and $\c(x')=nk+1$.
\label{lem:tweak}
\end{lem}

\begin{proof}
\def\c{\gamma}
Let $\{x_1,x_2,\ldots,x_k\}$ and $\{y_2,\ldots,y_{k}\}$ be the partite sets of~$H_k$,
and let $e_i$ be the half-edge in $H_k$ incident with $x_i$.
We may assume $e=e_1$ and $e'=e_2$, thus $v=x_1$ and $v'=x_2$.
Let $L=[\ell_{ij}]$ be a Latin square of order $k$ with $\ell_{ij}\in\{1,2,\ldots,r\}$,
$\ell_{11}=1$, and $\ell_{21}=k$. We define a $k+1$--total colouring of $H_k$ from $L$ as 
in the proof of Lemma~\ref{lem:all0}.

We may now multiply all colours by $n$ to get an $(n(k+1)+1,n)$--total colouring $\c_0$ of $H_k$.
Note that in such colouring each occurrence of the colour $0$ may be replaced by $n(k+1)=-1$
and the colouring remains proper.
So we may ``tweak'' the colouring $\c_0$ by recolouring $e_1$ by the colour $-1$. We then
shift all colours by $1$ to get the desired colouring~$\c$.
Hence $\c(e)=\c(e_1)=-1+1=0$, $\c(e')=\c(e_2)=n0+1=1$,
$\c(v)=\c(x_1)=n\ell_{11}+1=n+1$, and $\c(v')=\c(x_2)=n\ell_{21}+1=nk+1$.
\end{proof}

\newcommand{\G}[2]{G_{#1,#2}}

By way of Lemma~\ref{lem:all0}, we now use the graph $H_k$ to construct
type~$2$ graphs with maximum degree~$k$.
Let $n$ be a positive integer and $H'_k$ be the graph obtained from $H_k$ by deleting $k-2$
of its half-edges.
Let $B_1,\ldots,B_n$ be all isomorphic to $H'_k$,
and let $f_i$ and $f'_i$ be the half-edges of $B_i$.
Let $B_0$ be a graph with one vertex $u$ and two half-edges $f_0$ and $f'_{n+1}$.
We construct the graph $\G{k}{n}$ from the disjoint union of $B_0,B_1,\ldots,B_n$  by joining
the half-edges $f_i$ and $f'_{i+1}$ to form an edge $e_i$ for each $0\le i\le n$.

\begin{thm}
Let $k\ge 2$. For all positive integers $n$, the graph $\G{k}{n}$ is type~$2$.
Moreover, $k+1<\chi''_c(\G{k}{n})\le k+1+\frac{1}{n}$.
\label{thm:lim}
\end{thm}

\begin{proof}
\def\c{c}
Let $k\ge 2$ and $n\ge 1$. 
Let $B_i, f_i, f'_i, e_i$ ($0\le i\le n$) be as in the definition of $\G{k}{n}$. Moreover, let
$x_i$ and $x'_i$ be the unique end vertices of $f_i$ and $f'_i$ respectively. In particular
$x_0=x'_0=u$. Suppose $c$ is a $k+1$--total colouring of $\G{k}{n}$. Then by Lemma~\ref{lem:all0}
for each $1\le i\le n$ we have $c(e_{i-1})=c(e_i)$. Thus $c(e_0)=c(e_n)$ which is a contradiction
since $e_0$ and $e_n$ are adjacent at~$u$. Therefore $\G{k}{n}$ is type~$2$, which means
$\chi''_c(\G{k}{n})>k+1$.

To prove $\chi''_c(\G{k}{n})\le k+1+\frac{1}{n}$,
we give an $(n(k+1)+1,n)$--total colouring of $\G{k}{n}$.
By Lemma~\ref{lem:tweak}, $B_1$ admits an $(n(k+1)+1,n)$--colouring $\gamma$ with
$\gamma(f_1)=1$ and $\gamma(f'_1)=0$. We use $\gamma$ to define a (partial)
$(n(k+1)+1,n)$--total colouring of $\G{k}{n}$ with $\c(e_0)=0$, $\c(e_1)=1$, $\c(x_1)=nk+1$ and $\c(x'_1)=n+1$.
We extend $c$ to $B_i$ by shifting the colouring of $B_1$ by $i-1$. Thus we have $\c(e_i)=i$ for all $0\le i\le n$, and
    $\c(x_i)=nk+i$ and
    $\c(x'_i)=n+i$
for all $1\le i\le n$. We define $\c(u)=2n+1$.

To show that $\c$ is a proper $(n(k+1)+1,n)$--total colouring of $\G{k}{n}$ we only
need to show that the new edges $e_i$ do not create any conflict.
Note that $\c(x_i)=nk+i=-n-1+i$ and $\c(x'_{i+1})=n+i+1$ are consistent so the edge $e_i$
is valid for $1\le i\le n-1$. For $e_0$ and $e_n$ we have $\c(x'_1)=n+1$, $\c(x_n)=nk+n=-1$ 
which are both consistent with $\c(u)=2n+1$.
Finally, $\c(e_0)=0$ and $\c(e_n)=n$ are consistent and
since $k\ge 2$, they are both consistent with $\c(u)=2n+1$.
\end{proof}

The following are immediate from Theorem~\ref{thm:lim}.

\begin{cor}
For all $k\ge2$,
$\displaystyle\lim_{n\to\infty}\chi''_c(\G{k}{n})=k+1$.
\label{cor:lim}
\end{cor}

\begin{cor}
For every $k\ge2$, the set $\S_k$ is infinite.
\label{cor:inftly}
\end{cor}

\section{An Improved Upper Bound}

Note that $\G{2}{n}$ is isomorphic to the cycle $C_{3n+1}$ which has circular total chromatic number
$3+\frac{1}{n}$ as proved in~\cite{HackmannKemnitz}. So the upper bound of Theorem~\ref{thm:lim} is tight
for $k=2$.
For $k\ge3$ on the other hand, this upper bound is not tight. For a better upper bound, we need a better tweaking
of a $k+1$--total colouring of~$H_k$ than the rather easy one done in Lemma~\ref{lem:tweak}.

\begin{lem}
\def\c{\gamma}
Let $k\ge 2$. Let $e$ and $e'$ be two half-edges of the graph $H_k$, and let $x$ and $x'$ be the
vertices of $H_k$ incident with $e$ and $e'$ respectively.
Then for every positive integer $q$, $H_k$ admits a $(q(k+1)+1,q)$--total colouring $\c$
such that $\c(e)=0$, $\c(e')=2$, $\c(x)=qk+1$, and $\c(x')=q+2$.
\label{lem:refine2n}
\end{lem}

\begin{proof}
\def\c{\gamma_0}
\def\cc{\gamma}
Let $\{x_1,x_2,\ldots,x_k\}$ and $\{y_2,\ldots,y_{k}\}$ be the partite sets of~$H_k$,
and let $e_i$ be the half-edge in $H_k$ incident with $x_i$.
We may assume $e=e_k$ and $e'=e_1$, thus $v=x_k$ and $v'=x_1$.
Let $L=[\ell_{ij}]$ be the back-circulant Latin square of order~$k$, namely
$1\le\ell_{i,j}\le k$ and $\ell_{ij}=i+j-1\mod k$.

We start from the $k+1$--total colouring of
$H_k$ obtained from $L$ as in the proof of Lemma~\ref{lem:all0}, and multiply all colours by $q$
to obtain a $(q(k+1)+1,q)$--total colouring $\c$ of~$H_k$.
Then $\c(y_1)=\cdots=\c(y_{r-1})=\c(e_1)=\cdots=\c(e_k)=0$, $\c(x_1)=q$, and $\c(x_k)=qk$.

Note that the colours used on each vertex and its incident edges (and half-edges) are $0,q,2q,\ldots,kq$. Thus there
is a slack of $1$ between $kq$ and $0$. So we may shift all the colours $0,q,\ldots,sq$ by $-1$ ($0\le s\le k$), or
shift all the colours $tq,\ldots,kq$ by $1$ ($0\le t\le k$) and the colouring remains valid at that vertex.

We define $\cc$ from $\c$ by shifting by $1$ at each vertex $x_i$ with $i\le k-1$,
all the colours greater than or equal to $\c(x_i)$.
Because of our choice of $L$, we have $\cc(x_iy_j)=\c(x_iy_j)+1$ when
$i+j\le k+1$ and $\cc(x_iy_j)=\c(x_iy_j)$ otherwise. Effectively, at each vertex $y_j$ all the colours greater than or equal to
$\c(x_1y_j)$ are shifted by $1$. Therefore $\cc$ is a proper $(q(k+1)+1,q)$--total colouring at all vertices.

The shifts of colours performed above enable us now to define $\cc(e_1)=1$ and $\cc(e_k)=-1$.
We finish by shifting the colours of all vertices, edges and half-edges by~$1$.
An example is illustrated in Figure~\ref{fig:dist2}.
\end{proof}

\begin{figure}[hbt]
\begin{center}
\includegraphics[scale=1.0]{figtotal-3.mps}
\end{center}
\caption{A $(21,4)$--total colouring of $H_4$ found in the proof of Lemma~\ref{lem:refine2n}, before the final shift of all colours.
\def\c{\gamma}The left column shows the colours $\c(x_i)$, the right column $\c(e_i)$, the top row $\c(y_j)$, and the middle block
$\c(x_iy_j)$. Modifications from the original (scaled) Latin square colouring are shaded.}
\label{fig:dist2}
\end{figure}

\begin{thm}
For $k\ge4$ and $n\ge1$ we have $$\chi''_c(\G{k}{n})\le k+1+\frac{1}{2n}.$$
\label{thm:improve}
\end{thm}

\begin{proof}
\def\c{\gamma}
Let $k\ge4$ and $n\ge1$. 
Let $B_i, f_i, f'_i, e_i, x_i, x'_i$ ($0\le i\le n$) be as in the proof of Theorem~\ref{thm:lim}.
We give a $(2n(k+1)+1,2n)$--total colouring of $\G{k}{n}$ by combining shifts of
the colourings of Lemma~\ref{lem:refine2n}.
By Lemma~\ref{lem:refine2n} with $q=2n$, we have a $(2n(k+1)+1,2n)$--total colouring of
$B_1$ in $\G{k}{n}$ so that $c(e_0)=0$, $c(e_1)=2$, $c(x'_1)=2nk+1$ and $c(x_1)=2n+2$.
For the block $B_i$ where $2\le i\le n$, we use a shift (by $2i-2$) of the colouring of $B_1$.
Thus for $1\le i\le n$ we have
$c(e_i)=2i$ and
    $c(x_i)=2n+2i$ and
    $c(x'_i)=2nk+2i-1$.
Note that the end vertices of $e_i$ receive colours $c(x_i)=2n+2i$ and $c(x'_{i+1})=2nk+2i+1$
which are consistent. At the vertex $u$ we have $c(e_0)=0$ and $c(e_n)=2n$ which are consistent.
Moreover, since $c(x'_1)=2kn+1$ and $c(x_n)=4n$, we may define $c(u)=6n$. This does not conflict any of
the adjacent colours since $k\ge 4$.
\end{proof}

Combining $n-1$ shifts of an $(8n-3,2n-1)$--total colouring of $H'_3$
with $c(e)=0$, $c(e')=2$ (obtained from Lemma~\ref{lem:refine2n}) and an $(8n-3,2n-1)$--total colouring of
$H'_3$ with $c(e)=0$, $c(e')=1$, $c(x)=2n$, $c(x')=4n-1$ (obtained by tweaking the proof of Lemma~\ref{lem:tweak}),
one can prove $\chi''_c(\G{3}{n})\le4+{1}/({2n-1})$.
We omit details of such proof here.

By computer aid, we verified that the upper bound of Theorem~\ref{thm:improve} is tight
for $\G{4}{n}$ where $n=1,2,3,4,5$ and $\G{5}{n}$ where $n=1,2$.
We also verified by computer aid that $\chi''_c(\G{3}{1})=9/2$ and $\chi''_c(\G{3}{n})=4+{1}/({2n-1})$
for $n=2,3,\ldots,10$. These results suggest that the upper bounds of this section could indeed be the
actual circular total chromatic numbers of the graphs~$\G{k}{n}$.

\section{Concluding Remarks}

Similar to circular edge colouring, the circular
total chromatic numbers of graphs seem to be sparse. By analogy to Conjecture~\ref{conj:zhu}
and based on our computational experiments, we conjecture the following.

\begin{conj}
The set $\S_k$ has no accumulation points other than $k+1$.
\end{conj}

Note that the TCC implies that $\S_k\subseteq[k+1,k+2]$, thus if the TCC is true, then every
$a\in\S\setminus\mathbb{Z}$ can only be realized as the circular total chromatic number of
a graph with maximum degree~$\lfloor a\rfloor-1$.

It is worth mentioning here that the graph $\G{k}{n}$ is bipartite when $n$ is odd.
Hence Corollary~\ref{cor:lim} and Corollary~\ref{cor:inftly} remain valid for bipartite graphs.
A natural question to be asked here is whether every member of $\S$ can be realized as
the  circular total chromatic number
of some bipartite graph.

\def\soft#1{\leavevmode\setbox0=\hbox{h}\dimen7=\ht0\advance \dimen7
  by-1ex\relax\if t#1\relax\rlap{\raise.6\dimen7
  \hbox{\kern.3ex\char'47}}#1\relax\else\if T#1\relax
  \rlap{\raise.5\dimen7\hbox{\kern1.3ex\char'47}}#1\relax \else\if
  d#1\relax\rlap{\raise.5\dimen7\hbox{\kern.9ex \char'47}}#1\relax\else\if
  D#1\relax\rlap{\raise.5\dimen7 \hbox{\kern1.4ex\char'47}}#1\relax\else\if
  l#1\relax \rlap{\raise.5\dimen7\hbox{\kern.4ex\char'47}}#1\relax \else\if
  L#1\relax\rlap{\raise.5\dimen7\hbox{\kern.7ex
  \char'47}}#1\relax\else\message{accent \string\soft \space #1 not
  defined!}#1\relax\fi\fi\fi\fi\fi\fi}


\begin{thebibliography}{1}

\bibitem{Behzad}
M. Behzad.
\newblock {\em Graphs and their Chromatic Numbers}.
\newblock PhD thesis, Michigan State University, 1965.

\bibitem{BodeKemnitzKlages}
J.-P. Bode, A. Kemnitz, and R. Klages.
\newblock Circular total colorings of some type-2 graphs.
\newblock {\em Congr. Numer.}, 189:129--137, 2008.

\bibitem{HackmannKemnitz}
A. Hackmann and A. Kemnitz.
\newblock Circular total colorings of graphs.
\newblock {\em Congr. Numer.}, to appear.

\bibitem{HackmannKemnitz:circulant}
A. Hackmann and A. Kemnitz.
\newblock Circular total colorings of cubic circulant graphs.
\newblock {\em J. Combin. Math. Combin. Comput.}, 49:65--72, 2004.

\bibitem{VizingTCC}
V.~G. Vizing.
\newblock Nekotorye nereshennye zadachi v teorii grafov.
\newblock {\em Uspekhi Mat. Nau}, XXIII(6):125--141, 1968.
\newblock (Russian).

\bibitem{Zhu:accumulation}
X. Zhu.
\newblock Personal communication.

\bibitem{Zhu-survey}
X. Zhu.
\newblock Circular chromatic number: a survey.
\newblock {\em Discrete Math.}, 229(1-3):371--410, 2001.

\bibitem{zhu:newsurvey}
X. Zhu.
\newblock Recent developments in circular colouring of graphs.
\newblock In {\em Topics in Discrete Mathematics}, volume~26 of {\em Algorithms
  and Combinatorics}, pages 497--550. Springer, 2006.

\end{thebibliography}
\end{document}